\documentclass[11pt,a4paper]{article}
\usepackage[T1]{fontenc}
\usepackage{lmodern}
\usepackage[margin=27mm]{geometry}
\usepackage{amsmath,amsthm}
\usepackage[hidelinks]{hyperref}
\newtheorem{theorem}{Theorem}[section]
\newtheorem{lemma}[theorem]{Lemma}
\newtheorem{corollary}[theorem]{Corollary}
\newtheorem{proposition}[theorem]{Proposition}
\newcommand{\email}[1]{\texttt{#1}}
\newcommand{\headers}[2]{\markboth{#2}{#1}}
\newenvironment{keywords}{\par\medskip\noindent\textbf{Keywords:}\ }{\par}
\newenvironment{AMS}{\par\medskip\noindent\textbf{2020 Mathematics Subject Classification:}\ }{\par}
\usepackage{amssymb}
\numberwithin{equation}{section}
\newcommand{\G}{\mathbf{G}}
\renewcommand{\phi}{\varphi}
\title{Local Fan-type obstructions and Ore-type conditions for rainbow Hamiltonicity in graph systems}
\author{Ailian Chen\and Hao Chen\thanks{Both authors are with the School of Mathematics and Statistics, Fuzhou University, Fuzhou 350108, China. Ailian Chen: \email{elian1425@fzu.edu.cn} (corresponding author); Hao Chen: \email{1577761811@qq.com}.}}
\headers{Local conditions for rainbow Hamiltonicity}{Ailian Chen and Hao Chen}

\hypersetup{
  pdftitle={Local Fan-type obstructions and Ore-type conditions for rainbow Hamiltonicity in graph systems},
  pdfauthor={Ailian Chen and Hao Chen},
  pdfsubject={Local rainbow Hamiltonicity conditions in graph systems},
  pdfkeywords={graph system, rainbow Hamilton cycle, rainbow distance, Fan-type condition, Ore-type condition, rotation-extension}
}

\date{22 October 2026}

\begin{document}
\maketitle

\begin{abstract}
Let $\G=(G_1,\ldots,G_n)$ be a graph system on a common $n$-vertex set,
with system degree $\delta(v)=\min_i d_{G_i}(v)$, common edge set
$E_{\cap}=\bigcap_iE(G_i)$, and rainbow distance $d^c$.
Suppose every member is $2$-connected. Under the local Fan condition
that $\delta(u),\delta(v)<n/2$ and $d^c(u,v)\le2$ imply
$uv\in E_{\cap}$, every system without a rainbow Hamilton cycle has
$E_{\cap}=\varnothing$ and $2S+\Gamma\le n^2-1$, where
$S=\sum_v\delta(v)$ and, for the sorted system degrees
$a_1\ge\cdots\ge a_n$ and $\Delta=a_1$,
$\Gamma=\sum_{j=2}^{\Delta+1}(\Delta-a_j)$.
Thus $2S+\Gamma\ge n^2$, or the presence of a common edge, guarantees
rainbow Hamiltonicity. An explicit infinite family meets the refined
criterion below the earlier aggregate threshold and has no common edge.
The local Ore implication with threshold $n$
also guarantees a rainbow Hamilton cycle. At threshold $n-1$, we prove
the existence of a rainbow cycle of length at least $n-1$ and a
prescribed-set closure theorem for near-spanning paths with endpoint
degree sum at least $n-1$. Complete bipartite examples show that the
local thresholds $n$ for Hamiltonicity and $n-1$ for the long-cycle
conclusion cannot be lowered uniformly. The proofs combine a
shared rotation lemma, critical neighborhood rigidity, and an
auxiliary-digraph count that retains degree and color surpluses.
\end{abstract}

\begin{keywords}
graph system, rainbow Hamilton cycle, rainbow distance, Fan-type condition,
Ore-type condition, rotation--extension
\end{keywords}

\begin{AMS}
05C38, 05C35
\end{AMS}

\section{Introduction}

All graphs are finite, simple, and undirected. A \emph{graph system} is an
indexed family $\G=(G_1,\ldots,G_m)$ of graphs on a common vertex set $V$;
we interpret index $i$ as color $i$. A graph $H$ with $V(H)\subseteq V$ and
$|E(H)|\le m$ is \emph{rainbow} in $\G$ if its edges can be assigned distinct
colors $i\in[m]$ so that every edge of color $i$ lies in $G_i$. If
$|E(H)|=m$, then $H$ is a $\G$-transversal. We write
$[n]=\{1,\ldots,n\}$.

Fan's theorem \cite{Fan1984} localizes Dirac's degree condition: every
$2$-connected graph $G$ of order $n$ is Hamiltonian whenever
\[
  \max\{d_G(u),d_G(v)\}\ge \frac n2
\]
for each pair $u,v$ at distance two. Local Ore-type conditions provide a
parallel classical theory; see
\cite{HasratianKhachatrian1990,AsratianBroersmaHeuvelVeldman1996,
AsratianSarkisian1996}. This raises a natural question for graph systems:
does an analogous localization remain valid when ordinary distance is
replaced by rainbow distance?

Aharoni, DeVos, Gonz\'alez Hermosillo de la Maza, Montejano and
\v{S}\'amal \cite{Aharoni2020} posed the rainbow Dirac conjecture for graph
systems. Joos and Kim \cite{Joos2020} subsequently proved the exact result:
if $\delta(G_i)\ge n/2$ for every $i\in[n]$, then any system of $n$ graphs on
$n$ vertices contains a rainbow Hamilton cycle. Cheng, Wang and Zhao
\cite{Cheng2021} established an asymptotic version and obtained further
cycle-spectrum results. Related developments include stability results for
transversal Hamilton cycles and paths \cite{ChengStaden2025}, exact near-Dirac
results for Hamilton paths and cycles \cite{ChengSunWangWei2026}, and the
survey of Sun, Wang and Wei \cite{SunWangWeiSurvey2026}. Li, Li and Li
\cite{LiLiLi2023} studied rainbow Hamilton paths, vertex-pancyclicity, and
panconnectivity under degree conditions, and identified rainbow analogues of
classical Hamiltonicity criteria, including Fan-type conditions, as a natural
direction for further study.

Rainbow distance also appears in recent work on path spectra. For vertices
$x,y$, the lower end of the rainbow panconnectivity range is governed by the
length of a shortest rainbow $x$--$y$ path; see Li, Li and Li
\cite{LiLiLi2023}, Sun, Wang and Wei \cite{SunWangWei2025Panconnected}, and
Ma, You and Zhang \cite{MaYouZhang2026}. In the present paper, rainbow
distance serves a different purpose: it determines the pairs on which the
Hamiltonicity condition is imposed, rather than the range of path lengths
required between a fixed pair.

From now on let $\G=(G_1,\ldots,G_n)$ be a system of $n$ graphs on a common
$n$-vertex set $V$, and define the \emph{system degree}
\[
  \delta(v)=\min_{i\in[n]}d_{G_i}(v),
  \qquad
  E_{\cap}=\bigcap_{i=1}^nE(G_i).
\]
The \emph{rainbow distance} $d^c(u,v)$ is the minimum length of a rainbow
$u$--$v$ path, with $d^c(u,v)=\infty$ if no such path exists. We consider the
local Fan-type condition
\begin{equation*}
  \delta(u),\delta(v)<\frac n2,
  \qquad d^c(u,v)\le2
  \quad\Longrightarrow\quad
  uv\in E_{\cap}.
  \tag{RF}\label{eq:rainbow-fan-condition}
\end{equation*}
The distance-one case is specific to graph systems: an edge contained in a
single member already has rainbow distance one, whereas \textup{(RF)} forces
an edge joining two low-system-degree vertices to be common to all members.
The appearance of $E_{\cap}$ is natural in rotation--extension arguments: a
common edge may receive any available color and can therefore absorb a color
released when a path edge is deleted. In the diagonal case
$G_1=\cdots=G_n=G$, rainbow distance is ordinary graph distance and
$E_{\cap}=E(G)$, so \textup{(RF)} reduces exactly to Fan's condition.

Write the system degrees in nonincreasing order as
\[
 a_1\ge a_2\ge\cdots\ge a_n,\qquad
 \Delta=a_1,\qquad S=\sum_{j=1}^n a_j,
\]
and define the nonnegative degree-sequence correction
\begin{equation}\label{eq:gamma-definition}
 \Gamma=\sum_{j=2}^{\Delta+1}(\Delta-a_j)
        =\Delta^2-\sum_{j=2}^{\Delta+1}a_j.
\end{equation}
The index range is valid because $\Delta\le n-1$.

Our first result refines the aggregate bound by retaining the degree
differences discarded in the auxiliary-digraph count. It also identifies
the common-edge obstruction left by \textup{(RF)}.

\begin{theorem}[Structural obstruction]\label{thm:structural}
Let $n\ge3$, and let $\G=(G_1,\ldots,G_n)$ be a graph system on a common
$n$-vertex set $V$. Suppose that every $G_i$ is $2$-connected, that $\G$
satisfies \textup{(RF)}, and that $\G$ contains no rainbow Hamilton cycle.
Then
\[
  E_{\cap}=\varnothing
  \qquad\text{and}\qquad
  2S+\Gamma\le n^2-1.
\]
\end{theorem}

In particular, $S\le(n^2-1)/2$. The stronger degree-sequence
consequence is the following.

\begin{corollary}[Degree-sequence Hamiltonicity]\label{cor:degree-sequence}
Suppose that every $G_i$ is $2$-connected and that $\G$ satisfies
\textup{(RF)}. If
\[
 2S+\Gamma\ge n^2,
\]
then $\G$ contains a rainbow Hamilton cycle.
Equivalently, it suffices that $S\ge\lceil(n^2-\Gamma)/2\rceil$.
\end{corollary}

Proposition~\ref{prop:strict-family} constructs, for every even $n\ge10$,
systems with $E_{\cap}=\varnothing$ that meet the new inequality at its
integer threshold while falling $\lfloor n/4\rfloor$ below the old
aggregate threshold. Thus the stronger condition has realizable cases
outside both earlier sufficient conditions.

The correction vanishes exactly when the largest $\Delta+1$ system
degrees are all $\Delta$. More simply, if the maximum system degree has
multiplicity $q$, then
\[
 \Gamma\ge\max\{0,\Delta+1-q\}.
\]
Thus $2S+\max\{0,\Delta+1-q\}\ge n^2$ is another sufficient condition.
Since $\Gamma\ge0$, we also recover the original scalar threshold.

\begin{corollary}[Aggregate-degree Hamiltonicity]\label{cor:total-degree-ham}
Let $\G=(G_1,\ldots,G_n)$ be a graph system on a common $n$-vertex set.
Suppose that every $G_i$ is $2$-connected and that $\G$ satisfies
\textup{(RF)}. If
\[
  \sum_{v\in V}\delta(v)\ge\left\lceil\frac{n^2}{2}\right\rceil,
\]
then $\G$ contains a rainbow Hamilton cycle.
\end{corollary}

The common-edge conclusion gives another corollary.

\begin{corollary}[Common-edge Hamiltonicity]\label{cor:common-edge-ham}
Let $\G=(G_1,\ldots,G_n)$ be as above. If $E_{\cap}\ne\varnothing$, then
$\G$ contains a rainbow Hamilton cycle.
\end{corollary}

The common-edge corollary also shows that our structural theorem genuinely
extends the classical Fan theorem. Indeed, for a diagonal system
$G_1=\cdots=G_n=G$, condition \textup{(RF)} is precisely Fan's condition and
$E_{\cap}=E(G)$. Since a $2$-connected graph has an edge,
Corollary~\ref{cor:common-edge-ham} immediately recovers Fan's theorem.

These two corollaries identify two immediate ways to rule out the structural
obstruction. We return in Section~\ref{sec:concluding} to the open question of
whether \textup{(RF)}, together with $2$-connectivity of every member, already
forces a rainbow Hamilton cycle.

For graph systems, Liu, Chen and Ma \cite{LiuChenMa2025} proved a transversal
Hamiltonicity theorem under a cross-layer Ore-type degree-sum condition. In
our notation, their result in particular yields Hamiltonicity under the global
condition
\begin{equation}
  \delta(u)+\delta(v)<n
  \quad\Longrightarrow\quad
  uv\in E_{\cap}
  \qquad(u\ne v).
  \label{eq:global-ore}
\end{equation}
Li and Luo \cite{LiLuo2025} obtained a Hamiltonian-connected analogue under a
layerwise Ore condition, while Li, Wang and Yan \cite{LiWangYan2026} proved
pancyclic consequences under a global Ore-type hypothesis. For Hamiltonicity,
the implication \eqref{eq:global-ore} need only be imposed on rainbow-close
pairs.

\begin{theorem}[Local Ore-type condition]\label{thm:local-ore}
Let $n\ge3$, and let $\G=(G_1,\ldots,G_n)$ be a graph system on a common
$n$-vertex set $V$. Suppose that every $G_i$ is $2$-connected and that, for
every pair of distinct vertices $u,v\in V$,
\[
  \delta(u)+\delta(v)<n,
  \qquad d^c(u,v)\le2
  \quad\Longrightarrow\quad
  uv\in E_{\cap}.
\]
Then $\G$ contains a rainbow Hamilton cycle.
\end{theorem}

In the diagonal case the theorem becomes the classical local Ore requirement
\[
  d_G(u)+d_G(v)\ge n
  \qquad\text{whenever }d_G(u,v)=2.
\]
This localization is strictly weaker than the global condition. Let $n\ge6$,
write $W=A\sqcup B$ with $|W|=n-2$ and $|A|,|B|\ge2$, let $W$ induce
$K_{n-2}$, and add vertices $x,y$ so that $N(x)=A$ and $N(y)=B$. Taking every
member of the system equal to this graph gives a $2$-connected diagonal
system. The pair $x,y$ has distance three and $d(x)+d(y)=n-2<n$, so the global
condition fails. Every other nonedge has distance two and degree sum at least
$n$, so the local Ore hypothesis holds.

Thus the local Ore hypothesis is genuinely weaker than its global counterpart,
even for diagonal systems. Rainbow distance is therefore essential: the
degree-sum condition is imposed only on pairs that already admit a short
rainbow connection.

To lower the Ore threshold, distinguish Hamiltonicity from
near-spanning closure. Define
\begin{equation*}
 \delta(u)+\delta(v)<n-1,\qquad d^c(u,v)\le2
 \quad\Longrightarrow\quad uv\in E_{\cap}.
 \tag{LO$^-$}\label{eq:critical-local-ore}
\end{equation*}

\begin{theorem}[Critical local Ore long-cycle theorem]\label{thm:ore-almost-spanning}
Let $n\ge3$. If every member is $2$-connected and $\G$ satisfies
\textup{(LO$^-$)}, then $\G$ contains a rainbow cycle of length at least
$n-1$ (a Hamilton cycle when $n=3$).
\end{theorem}

The critical endpoint statement is stronger than global long-cycle
existence: under \textup{(LO$^-$)} and the absence of a rainbow Hamilton
cycle, every rainbow path on $n-1$ vertices with endpoint system-degree
sum at least $n-1$ closes on its prescribed vertex set
(Theorem~\ref{thm:prescribed-closure}). Handling equality is essential
to the proof;
its missing-color neighborhoods yield adjacent twins and then a cut
vertex. Diagonal complete bipartite examples show that the local Ore
threshold $n-1$ for the long-cycle conclusion cannot be lowered uniformly.
They also show why this improvement does not give Hamiltonicity at
threshold $n-1$. Optimality of the endpoint threshold under the retained
local Ore hypothesis is not claimed.

The paper is organized as follows. Section~\ref{sec:rotation} develops
the recoloring and trapping tools and proves a rainbow Hamilton-path
theorem. Section~\ref{sec:almost-spanning} proves the Fan-type
almost-spanning cycle theorem. Section~\ref{sec:critical-ore} treats
the critical local Ore threshold $n-1$, including closure on a
prescribed vertex set and the sharp general long-cycle conclusion.
Section~\ref{sec:structural-proof} proves the structural obstruction and
the degree-sequence bound. Section~\ref{sec:local-ore-proof} deduces
Hamiltonicity at the local Ore threshold $n$.

\section{Rainbow rotation and extension}\label{sec:rotation}

\subsection{Recoloring and endpoint closure}

For a rainbow subgraph $H$, fix an injective map
$\phi:E(H)\to[n]$ such that $e\in E(G_{\phi(e)})$ for every $e\in E(H)$.
We call $(H,\phi)$ a \emph{colored rainbow subgraph}, and call a color in
$[n]\setminus\phi(E(H))$ \emph{missing}. For an edge
$e\in\binom{V}{2}$, write
$\;c(e)=\{i\in[n]:e\in E(G_i)\}\;$
for its set of available colors. When normalizing the coloring of a fixed
rainbow path or cycle, we may apply a global permutation to the member labels.
Such a permutation preserves $\delta(v)$, $E_{\cap}$, rainbow distance, and
the existence of rainbow paths and cycles.

Throughout, edges unaffected by a rotation retain their colors. When a
color-$i$ edge is deleted, a retained common edge may be recolored with color
$i$; colors assigned to all other new edges are stated explicitly. The next
lemma records this elementary recoloring principle.

\begin{lemma}[Released-color recoloring principle]
\label{lem:rotation-principle}
Let $(H,\phi)$ be a colored rainbow path or cycle. Let
$F\subseteq E(H)$ be a set of edges to be deleted, let
$B\subseteq E(H)\setminus F$ consist of common edges that are retained
but may be recolored, and let $A$ be a set of new edges such that
\[
  H'=(H-F)+A
\]
is again a path or a cycle. Put
\[
  M=[n]\setminus\phi(E(H)),
  \qquad
  R=\phi(F\cup B).
\]
If the edges in $A\cup B$ admit pairwise distinct legal colors from
$M\cup R$, then, after making those assignments and leaving every edge
of $E(H)\setminus(F\cup B)$ in its original color, $H'$ is rainbow.
In particular, every edge of $B\subseteq E_{\cap}$ may receive any
chosen color in $M\cup R$ that has not already been assigned to another
edge of $A\cup B$.
\end{lemma}

\begin{proof}
Because $(H,\phi)$ is rainbow, no color in $R=\phi(F\cup B)$ appears on
an edge of $E(H)\setminus(F\cup B)$, and no color in $M$ appears on $H$.
Thus the colors retained on $E(H)\setminus(F\cup B)$ are disjoint from
$M\cup R$. The colors assigned to $A\cup B$ are legal and pairwise distinct
by hypothesis, so $H'$ is rainbow.
\end{proof}

\medskip
\noindent The following two closure lemmas will be used repeatedly.

\begin{lemma}[Endpoint-sum closure]\label{lem2.1}
	Let $\G=(G_1,\ldots,G_n)$ be a graph system on a common vertex set
	of size $n\ge3$. Suppose that $\G$ contains a rainbow Hamilton path
	\[
	P=u_1u_2\cdots u_n
	\]
	containing an edge of $E_{\cap}$. If
	\[
	\delta(u_1)+\delta(u_n)\ge n,
	\]
	then $\G$ contains a rainbow Hamilton cycle.
\end{lemma}

\begin{proof}
	Suppose for a contradiction that $\G$ has no rainbow Hamilton cycle.
	Fix a rainbow coloring of $P$ and globally permute the member labels so that
	\[
	\phi(u_iu_{i+1})=i
	\qquad(i\in[n-1]).
	\]
	Let $u_ku_{k+1}\in E_{\cap}$. Recoloring this common edge with color
	$n$ makes color $k$ missing. Hence
	\[
	k,n\notin c(u_1u_n),
	\]
	for otherwise $P$ can be closed into a rainbow Hamilton cycle.
	
	Define
	\[
	I_1=\{i\in[n-2]:k\in c(u_1u_{i+1})\},
	\qquad
	I_2=\{i\in[2,n-1]:n\in c(u_iu_n)\}.
	\]
	The sets $I_1$ and $I_2$ are disjoint. Otherwise, let
	$j\in I_1\cap I_2$. If $j\ne k$, delete $u_ju_{j+1}$, assign color
$k$ to $u_1u_{j+1}$ and color $n$ to $u_ju_n$, and recolor the common
	edge $u_ku_{k+1}$ with the released color $j$. If $j=k$, simply delete
	$u_ku_{k+1}$ and add $u_1u_{k+1}$ in color $k$ and $u_ku_n$ in color
	$n$. Either construction gives a rainbow Hamilton cycle, a contradiction.
	
	Thus
	\[
	|I_1|+|I_2|\le n-1.
	\]
	Since $k,n\notin c(u_1u_n)$, every $G_k$-neighbor of $u_1$ is counted
	by $I_1$, and every $G_n$-neighbor of $u_n$ is counted by $I_2$.
	Hence
	\[
	\delta(u_1)+\delta(u_n)
\le d_{G_k}(u_1)+d_{G_n}(u_n)
=|I_1|+|I_2|
	\le n-1,
	\]
	contrary to the hypothesis.
\end{proof}

\begin{lemma}[Almost-spanning closure]\label{lem2.2}
	Let $n\ge4$, and suppose that $\G$ contains no rainbow Hamilton cycle.
	If
	\[
	P=v_1v_2\cdots v_{n-1}
	\]
	is a rainbow path satisfying
	\[
	\delta(v_1)+\delta(v_{n-1})\ge n,
	\]
	then $V(P)$ spans a rainbow cycle of length $n-1$.
\end{lemma}

\begin{proof}
	Suppose for a contradiction that there is no rainbow cycle on $V(P)$.
	Let $v_n$ be the unique vertex outside $P$. Globally permute the member
	labels so that
	\[
	\phi(v_iv_{i+1})=i
	\qquad(i\in[n-2]).
	\]
	Thus colors $n-1$ and $n$ are missing.
	
	If
	\[
	v_1v_n\in E(G_n)
	\quad\text{and}\quad
	v_{n-1}v_n\in E(G_{n-1}),
	\]
	then these two edges together with $P$ form a rainbow Hamilton cycle.
	Hence, by symmetry, we may assume
	\[
	v_1v_n\notin E(G_n).
	\]
	Since there is no rainbow cycle on $V(P)$ and both colors $n-1$ and $n$
	are missing from $P$, we also have
	\[
	v_1v_{n-1}\notin E(G_{n-1})\cup E(G_n).
	\]
	
	Define
	\[
	I_1=\{i\in[n-3]:n\in c(v_1v_{i+1})\},
	\qquad
	I_2=\{i\in[2,n-2]:n-1\in c(v_iv_{n-1})\}.
	\]
	The sets $I_1$ and $I_2$ are disjoint. If $j\in I_1\cap I_2$, then
	deleting $v_jv_{j+1}$ and adding
	$v_1v_{j+1}$ and $v_jv_{n-1}$, with colors $n$ and $n-1$,
	respectively, gives a rainbow cycle on $V(P)$, a contradiction.
	
	Therefore
	\[
	|I_1|+|I_2|\le n-2.
	\]
	By the two displayed nonadjacencies above, every $G_n$-neighbor of
	$v_1$ is counted by $I_1$. Likewise, the only possible
	$G_{n-1}$-neighbor of $v_{n-1}$ not counted by $I_2$ is $v_n$. Hence
	\[
	\delta(v_1)+\delta(v_{n-1})
	\le d_{G_n}(v_1)+d_{G_{n-1}}(v_{n-1})
	\le |I_1|+|I_2|+1
	\le n-1,
	\]
	contrary to the hypothesis.
\end{proof}

\subsection{Trapped endpoints}

\begin{lemma}[Trapped endpoints force a vertex cut]\label{lem:trapped}
Let $P=x_1\cdots x_m$, $m\ge3$, be rainbow with edge colors
$1,\ldots,m-1$, and use the palette $[m]$.
Put $S=V(P)$ and let $L\subseteq M\subseteq V$, with $x_1\in L$.
Suppose that distinct vertices of $L$ and $M$ at rainbow distance at most two
are joined by a common edge. Assume the following for rainbow paths
$Q$ on $S$ ending at $x_m$ and using $m-1$ colors from this palette.
\begin{enumerate}
\item If $Q$ starts at $w$ and misses $r$, then
$N_{G_r}(w)\subseteq S\setminus\{x_m\}$.
\item If $x_1$ is internal in $Q$, with neighbors $x_2$ and
$q\in N_{G_1}(x_1)\cup N_{G_m}(x_1)$, then its first endpoint $w$
lies in $M$; if $w\notin N_{G_m}(x_1)$, then $w\in L$.
\end{enumerate}
Then $G_1$ and $G_m$ cannot both be $2$-connected.
\end{lemma}
\begin{proof}
Assume both are $2$-connected. Condition (1) for $P$ gives
$N_{G_m}(x_1)\subseteq\{x_2,\ldots,x_{m-1}\}$.
Minimum degree two forces $m\ge4$ and provides a maximal index
$K\in[3,m-1]$ with $x_1x_K\in G_m$.
Rotate at this chord. Condition (2) puts the exposed endpoint
$x_{K-1}$ in $M$, so the rainbow path $x_1x_Kx_{K-1}$ makes
$x_1x_{K-1}$ common. Descending in the same way gives
\[
 x_i\in M,\quad x_1x_i\in E_\cap\quad(2\le i<K),
 \qquad x_{K+1}\notin M.                              \tag{eq:trapped-star}\label{eq:trapped-star}
\]
The last assertion follows from maximality of $K$ using
$x_1x_Kx_{K+1}$.

For $2\le i<K$, rotate to $x_i\cdots x_1x_{i+1}\cdots x_m$,
use $m$ on $x_1x_{i+1}$, and recolor $x_1x_2$ with $i$.
For $i=1$, just recolor $x_1x_2$ with $m$.
These paths miss $1$, so (1) excludes $G_1$-edges from
$\{x_1,\ldots,x_{K-1}\}$ to $(V\setminus S)\cup\{x_m\}$.
Connectivity of $G_1-x_K$ forces $x_jx_s\in G_1$ with
$1\le j<K<s\le m-1$; in particular $K\le m-2$.
Consider
\[
 R(j,s)=x_{s-1}\cdots x_{j+1}x_1\cdots x_jx_s\cdots x_m.
\]
For $j\ge2$, color its two joining edges with $m,1$ and recolor
$x_1x_2$ with $j$. For $j=1$, use $1$ on $x_1x_s$ and $m$ on $x_1x_2$.
The path misses $s-1$ and satisfies (2).
If $s\ge K+2$, its first endpoint lies outside $N_{G_m}(x_1)$,
hence in $L$. Since $x_j\in M$, the two-step path through $x_s$
makes $x_jx_{s-1}$ common. Replace the crossing by this edge in color
$1$ and repeat downwards. This exposes $x_{K+1}$, contradicting
\eqref{eq:trapped-star}. Thus $s=K+1$. The same path puts $x_K\in M$, so
$x_1x_K$ is common. The common star now extends through $K$.

For $2\le i<K$, rotate at the common edge $x_1x_{i+1}$ using the
released color $i$; the resulting path starts at $x_i$ and misses $m$.
For $i=1$, use $P$ itself.
For $i=K$, use $R(j,K+1)$ instead: when $j\ge2$, color
$x_jx_{K+1},x_1x_{j+1},x_1x_2$ with $1,K,j$; when $j=1$, use $1,K$
on the crossing and $x_1x_2$. It too misses $m$.
Thus (1) excludes $G_m$-edges from the whole block
$\{x_1,\ldots,x_K\}$ to $(V\setminus S)\cup\{x_m\}$.

There is also no $G_m$-edge $x_ix_p$ with $i<K$ and
$K+2\le p\le m-1$. The case $i=1$ follows from maximality of $K$.
For $i\ge2$, such an edge makes
\[
 T(i,p)=x_{p-1}\cdots x_{i+1}x_1\cdots x_ix_p\cdots x_m
\]
rainbow: use $m$ on $x_ix_p$ and $i$ on the common chord
$x_1x_{i+1}$. Its first endpoint lies in $L$ by (2), and $x_i\in M$.
Descending the crossing in color $m$ therefore gives the same
contradiction to \eqref{eq:trapped-star}.

Connectivity of $G_m-x_{K+1}$ must now provide
$x_Kx_t\in G_m$ for $K+2\le t\le m-1$. The path
\[
 F(t)=x_{t-1}\cdots x_{K+1}x_j\cdots x_1x_{j+1}\cdots x_Kx_t\cdots x_m
\]
is rainbow: use $1,m$ on the two crossings; if $j\ge2$, use $j,K$
on $x_1x_2,x_1x_{j+1}$; if $j=1$, use $K$ on $x_1x_2$.
It misses $t-1$. The neighbors of $x_1$ are $x_2,x_{j+1}$ when
$j\ge2$, or $x_2,x_{K+1}$ when $j=1$, so (2) applies.
Its first endpoint is in $L$ and $x_K\in M$; descending this last
crossing once again exposes the forbidden $x_{K+1}$. Contradiction.
\end{proof}

\subsection{A spanning rainbow path}\label{sec:longest-path}

\begin{theorem}[Spanning rainbow path]\label{thm2.3}
Suppose that every $G_i$ is $2$-connected and that $\G$ satisfies either
\textup{(RF)} or \textup{(LO$^-$)}. Then $\G$ contains a rainbow Hamilton path.
\end{theorem}
\begin{proof}
Suppose otherwise, and choose a longest rainbow path $P=x_1\cdots x_m$,
$m<n$, maximizing its endpoint system-degree sum. Minimum degree two
gives $m\ge3$. Let $S=V(P)$. A rainbow cycle on $S$, together with an
edge crossing from $S$ to $V\setminus S$ in a connected missing layer,
would give a longer rainbow path. Thus no such cycle exists.
For two distinct missing colors $a,b$, the sets
\[
 A=\{j\in[m-1]:x_1x_{j+1}\in E(G_a)\},\qquad
 B=\{j\in[m-1]:x_jx_m\in E(G_b)\}
\]
are disjoint. Every neighbor in either missing layer lies in $S$ by
maximality, and the closing edge is absent. Consequently
\[
 \delta(x_1)+\delta(x_m)\le |A|+|B|\le m-1\le n-2.
\]
In the case \textup{(RF)}, orient $P$ so that $\delta(x_1)<n/2$ and set
$L=M=\{v:\delta(v)<n/2\}$. In the case \textup{(LO$^-$)}, put
$h=\lfloor n/2\rfloor$, orient so that $\delta(x_1)<h$, and set
$L=M=\{v:\delta(v)<h\}$. In either case, distinct vertices of $L$ at
rainbow distance at most two are joined by a common edge. For the second
case this follows from $2(h-1)<n-1$.

Apply Lemma~\ref{lem:trapped} with the palette consisting of the path
colors and one missing color. Its condition~\textup{(1)} follows from
nonextension and the absence of a cycle on $S$. Every rainbow path on $S$
ending at $x_m$ is also longest; endpoint extremality bounds its first
endpoint degree by $\delta(x_1)$, giving condition~\textup{(2)}.
The resulting vertex cut contradicts $2$-connectivity.
\end{proof}

\section{Almost-spanning cycles under the Fan condition}\label{sec:almost-spanning}

\begin{theorem}[Almost-spanning rainbow cycle]\label{thm2.4}
If every $G_i$ is $2$-connected and $\G$ satisfies \textup{(RF)}, then
$\G$ contains a rainbow cycle of length at least $n-1$.
\end{theorem}
\begin{proof}
For $n\in\{3,4\}$, every member has minimum degree at least $n/2$, and
the transversal Dirac theorem~\cite{Joos2020} gives a rainbow Hamilton
cycle. Suppose $n\ge5$ and there is neither a rainbow $C_{n-1}$ nor $C_n$.
By Theorem~\ref{thm2.3}, choose a rainbow Hamilton path
$P=v_1\cdots v_n$ maximizing $\delta(v_1)+\delta(v_{n-1})$.
Write $r$ for its missing color and $\ell$ for its last-edge color.
Lemma~\ref{lem2.2} gives
$\delta(v_1)+\delta(v_{n-1})\le n-1$, so at least one of these
vertices has system degree below $n/2$.

If $\delta(v_1)<n/2$, apply Lemma~\ref{lem:trapped} to the prefix
$P-v_n$, with its path colors and $r$ as palette and
$L=M=\{w:\delta(w)<n/2\}$. A missing-color edge to the fixed last
vertex closes a $C_{n-1}$, and one to $v_n$, after appending the
unchanged last edge in color $\ell$, closes a $C_n$. This proves
condition~\textup{(1)}. Appending that last edge to any candidate
prefix path and using the choice of $P$ proves condition~\textup{(2)}.
This is impossible. Hence
$\delta(v_1)\ge n/2$ and $\delta(v_{n-1})<n/2$.

If $\delta(v_n)<n/2$, the last edge of $P$ is common by \textup{(RF)}.
Reverse $P$ and apply the same lemma to its prefix. Condition~\textup{(1)}
is unchanged. Every path appearing in condition~\textup{(2)}, after the
fixed final edge is appended, contains the protected common first edge
and ends at $v_1$. Lemma~\ref{lem2.1} forces its other endpoint below
$n/2$, as required. This too is impossible.

Thus both endpoints of $P$ have system degree at least $n/2$. Put
$h=\lceil n/2\rceil$. If $v_iv_n\in E(G_r)$ and $i\le n-2$, the right
rotation
\[
 v_1\cdots v_i v_n v_{n-1}\cdots v_{i+1}
\]
has penultimate vertex $v_{i+2}$. Endpoint extremality gives
$\delta(v_{i+2})\le\delta(v_{n-1})<n/2$. The indices $i=1,2$ close
cycles of lengths $n,n-1$, respectively, while $i=n-2$ would make
$v_n$ low. Therefore
\[
 I=\{i+2:v_iv_n\in E(G_r),\ 3\le i\le n-3\}\subseteq[5,n-1],
 \qquad |I|\ge h-1.
\]
For $n=5$ the neighbor restriction already contradicts minimum degree
two. For $n\ge6$, since $h-1>\lceil(n-5)/2\rceil$, the set $I$
contains consecutive indices. Their path edge joins two low vertices
and is common by \textup{(RF)}. The original $P$ now contradicts
Lemma~\ref{lem2.1}.
\end{proof}

\section{Critical local Ore closure}\label{sec:critical-ore}

This section lowers the near-spanning endpoint threshold from $n$ to
$n-1$. Its hypothesis is \textup{(LO$^-$)}; no implication between this
condition and \textup{(RF)} is assumed for odd $n$.

\subsection{Equality in endpoint counting}
\begin{lemma}[Endpoint counting]\label{lem:critical-count}
The following statements require neither connectivity nor \textup{(LO$^-$)}.
\begin{enumerate}
\item If a rainbow Hamilton path $H$ contains a common edge and there is
no rainbow $C_n$, then its endpoint degree sum is at most $n-1$.
At equality, both terminal edges belong to the missing color of $H$.
\item Let $P=x_1\cdots x_m$, $m=n-1\ge3$, omit $z$ and miss $a,b$.
Suppose there is no rainbow $C_n$ and no rainbow cycle on $V(P)$.
Write $u=x_1$, $v=x_m$. Then $\delta(u)+\delta(v)\le n-1$.
At equality, orient $P$ so that
\[
 zu\in G_a\cap G_b,\qquad zv\notin G_a\cup G_b.          \tag{eq:critical-attachment}\label{eq:critical-attachment}
\]
For $\{s,t\}=\{a,b\}$, the sets
\[
 A_s=\{j:ux_{j+1}\in G_s\},\quad
 B_t=\{j:x_jv\in G_t\}
\]
partition $[m-1]$, and $d_{G_s}(u)=\delta(u)$,
$d_{G_s}(v)=\delta(v)$. The first and last edges of $P$ lie in
$G_a\cap G_b$. At each endpoint its neighborhoods in colors $a,b$
and either original terminal-edge color are identical.
\end{enumerate}
\end{lemma}
\begin{proof}
For (1), let $r$ be missing and $c$ be the color of a common edge.
An endpoint closing edge in $r$ or $c$ gives $C_n$ after recoloring.
Otherwise define
$A_c=\{j:y_1y_{j+1}\in G_c\}$ and
$B_r=\{j:y_jy_n\in G_r\}$ for $H=y_1\cdots y_n$.
An intersection gives $C_n$: delete $y_jy_{j+1}$, use $c,r$ on the
two chords, and recolor the retained common edge with the deleted color.
Thus $A_c,B_r$ are disjoint and their sizes count the two full degrees.
At degree sum $n-1$ they partition $[n-1]$. The last position forces
$y_{n-1}y_n\in G_r$; reverse $H$ for its first edge.

For (2), the same two-chord argument makes $A_s,B_t$ disjoint.
Put $\alpha=\mathbf1_{zu\in G_s}$,
$\beta=\mathbf1_{zv\in G_t}$. No Hamilton cycle gives $\alpha+\beta\le1$.
Consequently
\[
 \delta(u)+\delta(v)
 \le d_{G_s}(u)+d_{G_t}(v)
 =|A_s|+|B_t|+\alpha+\beta\le(m-1)+1=n-1.
\]
At equality every inequality is tight, for both ordered pairs $s,t$.
The first and last partition positions put both terminal edges in both
missing colors. If $c$ is the first-edge color, recolor that edge with
$a$ or $b$. The same argument now applies to all three missing pairs
$\{a,b\},\{a,c\},\{b,c\}$.
For each $j$, their partition identities force
\[
 \mathbf1_{ux_{j+1}\in G_a}=\mathbf1_{ux_{j+1}\in G_b}
 =\mathbf1_{ux_{j+1}\in G_c},
\]
and the corresponding three indicators at $v$ are equal.
The outside identities $\alpha+\beta=1$ give the same equalities at $z$;
hence the two orientations in \eqref{eq:critical-attachment}. Recolor the last edge instead
to include its original color as well.
\end{proof}

\begin{lemma}[The omitted vertex inherits the endpoint neighborhood]\label{lem:critical-twins}
In the equality case of Lemma~\ref{lem:critical-count}(2), oriented as in \eqref{eq:critical-attachment},
put $p=\delta(u)$, $q=\delta(v)$, $t=\delta(z)$, so $p+q=n-1$.
Under \textup{(LO$^-$)},
\[
 t\le p,\qquad N_{G_s}(z)\setminus\{u\}\subseteq N_{G_s}(u)
 \quad(s=a,b).                                        \tag{eq:critical-twins}\label{eq:critical-twins}
\]
If $t=p$, then $d_{G_s}(z)=p$ and $u,z$ are adjacent twins in $G_s$:
$N_{G_s}(u)\setminus\{z\}=N_{G_s}(z)\setminus\{u\}$.
Connectivity is not required in this lemma.
\end{lemma}
\begin{proof}
First suppose $t>p$, and write the path in the reverse orientation as
$x_1\cdots x_m$, with $\delta(x_1)=q$, $\delta(x_m)=p$.
In $H=x_1\cdots x_mz$, use $a$ on $x_mz$.
With $H_i=x_i$ for $i\le m$ and $H_{m+1}=z$, the sets
$\{j:x_1H_{j+1}\in G_a\}$ and $\{j:H_jz\in G_b\}$ in $[m]$
have total size at least $q+t\ge n>m$.
Their intersection yields $2\le j\le m-2$ with
$x_1x_{j+1}\in G_a$ and $zx_j\in G_b$.
Rotating $P$ at the first chord exposes $x_j$ and still misses $b$.
Lemma~\ref{lem:critical-count}(2) gives $\delta(x_j)\le q$; equality and \eqref{eq:critical-attachment}
are incompatible with both $zx_j,zx_m\in G_b$.
Thus $\delta(x_j)<q$, and $x_m,z,x_j$ makes $x_mx_j$ common by \textup{(LO$^-$)}.
Writing $c_i$ for the original color of $x_ix_{i+1}$, the path
\[
 x_{m-1}\cdots x_{j+1}x_1\cdots x_jx_mz
\]
is rainbow: its new edges use $a,c_{m-1},b$, respectively.
It contains a common edge, so Lemma~\ref{lem:critical-count}(1) gives
$\delta(x_{m-1})+t\le n-1$.
Hence $p+\delta(x_{m-1})<n-1$, making $x_{m-1}x_m$ common.
The original $H$ now contradicts Lemma~\ref{lem:critical-count}(1), since its
endpoint sum is $q+t\ge n$. This proves $t\le p$.

Return to the orientation $u=x_1$, $v=x_m$. If $j\in B_b$ and
$zx_{j+1}\in G_a$, rotate at $x_jv$ to expose $w=x_{j+1}$.
Then $\delta(w)\le q$. Equality contradicts \eqref{eq:critical-attachment} for this rotated
path. Otherwise $t+\delta(w)<p+q=n-1$, making $zw$ common.
But
\[
 u,z,w,x_{j+2},\ldots,v,x_j,x_{j-1},\ldots,x_2,u
\]
is a rainbow Hamilton cycle: the three new edges use $a,c_j,b$.
Thus $zw\in G_a$ implies $j\in A_a$, proving the inclusion.
Since $|A_a|=p-1$, it also gives $d_{G_a}(z)\le p$.
If $t=p$, equality forces the twin identity. Exchange $a,b$.
\end{proof}

\begin{lemma}[Critical cycle domination]\label{lem:critical-insertion}
Let $C$ be a rainbow cycle on $V\setminus\{v\}$ containing a common
edge. If there is no rainbow Hamilton cycle and
$\delta(v)\ge(n-1)/2$, then $n-1$ is even,
$\delta(v)=(n-1)/2$, and $d^c(v,x)\le2$ for every $x\in V(C)$.
This statement does not require \textup{(LO$^-$)} or connectivity.
\end{lemma}
\begin{proof}
Write $C=y_1\cdots y_my_1$, $m=n-1$, and let $r$ be its missing
color and $c$ the color of a common edge. Put
$A=N_{G_r}(v)$ and $B=N_{G_c}(v)$, viewed as subsets of the cyclic
vertex sequence. If $y_i\in A$ and $y_{i+1}\in B$, insert $v$ between
them using $r,c$. If the common edge is retained, recolor it with the
deleted edge's color. This is a rainbow Hamilton cycle. The reverse
ordered pair is equally impossible.
Consequently, with indices modulo $m$,
\[
 \mathbf1_{y_i\in A}+\mathbf1_{y_{i+1}\in B}\le1,
 \qquad
 \mathbf1_{y_i\in B}+\mathbf1_{y_{i+1}\in A}\le1.
\]
Since $|A|,|B|\ge\delta(v)\ge m/2$, equality holds in each inequality
and $|A|=|B|=m/2$. The indicators of $A$ are $2$-periodic, so $A=B$
is one parity class of the even cycle. Every vertex in $A$ is adjacent
to $v$ in color $r$. A vertex outside $A$ has a cycle neighbor in $A$;
its cycle edge has color different from $r$, giving a rainbow two-step
path to $v$. This proves the assertion.
\end{proof}

\subsection{Closure on a prescribed vertex set}

\begin{theorem}[Prescribed-set closure]\label{thm:prescribed-closure}
Let $n\ge4$. Suppose that every member is $2$-connected, that $\G$
satisfies \textup{(LO$^-$)}, and that there is no rainbow Hamilton cycle.
If $P=x_1\cdots x_{n-1}$ is a rainbow path with
\[
 \delta(x_1)+\delta(x_{n-1})\ge n-1,
\]
then its originally specified vertex set $V(P)$ spans a rainbow cycle.
\end{theorem}
\begin{proof}[Proof of Theorem~\ref{thm:prescribed-closure}]
Suppose $S=V(P)$ does not span a rainbow cycle, while there is no
rainbow Hamilton cycle. Lemma~\ref{lem:critical-count}(2) reduces the endpoint
sum to equality. Orient
\[
 P=x_1\cdots x_m,\quad m=n-1,
 \quad u=x_1,\quad v=x_m,\quad V\setminus S=\{z\},
\]
as in \eqref{eq:critical-attachment}, with missing colors $a,b$, and put
$p=\delta(u)$, $q=\delta(v)$, $t=\delta(z)$.
Thus $p+q=m$ and $p,q\ge2$, so $n\ge5$.

\emph{The attached endpoint is a smaller endpoint: $p\le q$.}
If $p>q$, apply Lemma~\ref{lem:trapped} to the reverse path
$v\cdots u$, using the palette consisting of its path colors and $b$.
Color $a$ stays missing. For every candidate path $Q$ ending at $u$,
its first endpoint $w$ has degree at most $q$, by Lemma~\ref{lem:critical-count}(2).
A missing-color edge from $w$ to $u$ closes a cycle on $S$; one from
$w$ to $z$, together with $zu$ in color $a$, closes a Hamilton cycle.
Thus the no-escape condition holds. Set
$L=M=\{w:\delta(w)\le q\}$; \textup{(LO$^-$)} supplies common edges since
$2q<p+q=m$. The fan lemma gives a contradiction.

\emph{The omitted vertex has the same degree: $t=p$.}
Lemma~\ref{lem:critical-twins} gives $t\le p$. Suppose $t<p$.
Then $zu$ is common, since $t+p< p+q=m$.
Also $d^c(z,v)>2$: otherwise \textup{(LO$^-$)}, using $t+q<m$, would make
$zv$ common and close $P$ through $z$ with colors $a,b$.
Let $\ell$ be the last-edge color of $P$. Endpoint-neighborhood rigidity
and Lemma~\ref{lem:critical-twins} give
\[
 zq'\in G_a\cup G_b,\quad q'\ne u
 \quad\Longrightarrow\quad uq'\in G_\ell.             \tag{S'}\label{eq:critical-shortcut}
\]
Prepend $z$ to $P$ in color $a$, making a Hamilton path $H$ missing $b$.
Apply Lemma~\ref{lem:trapped} to its prefix $H-v$, protecting the common edge
$zu$, with the prefix colors plus $b$ as palette. The color $\ell$
lies outside this palette and colors the unchanged final edge of $H$.
Take
\[
 L=\{w:\delta(w)<p\},\qquad M=\{w:\delta(w)\le p\}.
\]
The common-edge condition follows from $2p-1<m$ and \textup{(LO$^-$)}.

We verify the no-escape condition for a candidate prefix $Q$ whose
missing palette color is $r$. An edge from its first endpoint $w$ to $v$ in
color $r$ closes a Hamilton cycle after appending the last edge of $H$.
An edge to the fixed last vertex of $Q$ in color $r$ closes a cycle on
$V\setminus\{v\}$ containing the protected common edge $zu$.
Since $q\ge m/2$, Lemma~\ref{lem:critical-insertion} would imply
$d^c(z,v)\le2$, the same contradiction. This verifies condition (1)
using only the absence of a cycle on the original set $S$.

For condition (2), append the final edge of $H$ to $Q$. The resulting
Hamilton path contains $zu$, so Lemma~\ref{lem:critical-count}(1) gives
$\delta(w)+q\le m$, hence $\delta(w)\le p$.
If equality holds, recolor its terminal edge with its missing color,
releasing $\ell$. Delete $z$ and shortcut its two neighbors in color
$\ell$, using \eqref{eq:critical-shortcut}. This produces a rainbow path
on the original $S$, ending at $w,v$, with endpoint degree sum $m$.
Since $zv$ is absent in every color, \eqref{eq:critical-attachment} forces $zw$ to exist.
As $t+p<m$, it is common; in particular $w\in N_{G_b}(z)$.
Thus an exposed endpoint outside $N_{G_b}(z)$ has degree below $p$,
which is exactly condition (2). The fan lemma contradicts $t<p$.
We conclude $t=p$. The missing-layer twin identities now apply to $u,z$.

\emph{A predecessor-closed neighborhood gives a vertex cut.}
Put $B=N_{G_b}(z)$, so $|B|=p$, $u\in B$, $v\notin B$.
If $x_r\in B\setminus\{u\}$, twinhood and endpoint-neighborhood rigidity
give $ux_r\in G_a$. Rotate at this chord, exposing $w=x_{r-1}$ and
still missing $b$. Lemma~\ref{lem:critical-count}(2) gives $\delta(w)\le p$.
If equality holds, \eqref{eq:critical-attachment} gives $zw\in G_b$.
If $\delta(w)<p$, the two-step paths $u,x_r,w$ and $z,u,w$ make
$uw,zw$ common, since $2p-1<m$. In either case $w\in B$.
For $r=2$ the rotation is simply a first-edge recoloring.
Thus $B$ is predecessor-closed along $P$, and
\[
\begin{gathered}
 B=\{x_1,\ldots,x_p\},\qquad K=\{z,x_1,\ldots,x_{p-1}\},\\
 h=x_p,\qquad R=\{x_{p+1},\ldots,x_m\}.
\end{gathered}                           \tag{K}\label{eq:critical-blocks}
\]
Here $|K|=p$, $|R|=q$, and $v\in R$.
Every $x_i$, $i<p$, is exposed by the rotation at $ux_{i+1}$,
so its degree is at most $p$.
If its degree is $p$, the rotated path and Lemma~\ref{lem:critical-twins} give
\[
 N_{G_b}(x_i)=(B\setminus\{x_i\})\cup\{z\}.
\]
Consequently neither $z$ nor any degree-$p$ vertex of $K$ has a
$G_b$-neighbor in $R$.

If there are no degree-below-$p$ vertices in $K$, this already makes
$h$ a cut vertex of $G_b$. For every such low vertex $w$,
the rotation just described makes $zw$ common. For every
$y\in K\setminus\{w,z\}$, the two-step path $w,z,y$ and \textup{(LO$^-$)} make
$wy$ common, since $\delta(y)\le p$. Thus $w$ has $p-1$ common
neighbors in $K$, and $\delta(w)=p-1$.
Fix one such $w$. It is also an endpoint of a rainbow path on $S$ ending at $v$ and
missing $b$. Hence $d^c(w,v)>2$: otherwise \textup{(LO$^-$)}, using
$\delta(w)+q=m-1<m$, gives a common closing edge on $S$.
Since $w$ is joined commonly to every other vertex of $K$, this implies
that $v$ has no neighbor in $K$ in any layer. But $\delta(v)=q$ and
the set outside $K\cup\{v\}$ has exactly $q$ vertices. Therefore
\[
 N_{G_i}(v)=(R\setminus\{v\})\cup\{h\}
 \quad\text{for every color }i.                       \tag{V}\label{eq:critical-vstar}
\]
In particular these edges are common. Rotating $P$ at $vx_i$ in the
missing color $b$, for $p\le i\le m-1$, exposes $x_{i+1}$ while
keeping $u$ fixed. Lemma~\ref{lem:critical-count}(2) gives
$\delta(y)\le q$ for every $y\in R$.

Now let $w$ be any low vertex of $K$ and $y\in R$.
If $wy\in G_b$, \textup{(LO$^-$)} makes it common, because
$\delta(w)+\delta(y)\le(p-1)+q=m-1<m$.
This adds a $p$th common neighbor to the $p-1$ already in $K$,
contradicting $\delta(w)=p-1$.
Thus no $G_b$-edge joins $K$ and $R$.
Deleting $h$ disconnects $G_b$, contrary to its $2$-connectivity.
The original $S=V(P)$ must therefore span a rainbow cycle.
\end{proof}

\subsection{The sharp near-spanning threshold}

\begin{lemma}\label{lem:ore-high}
Assume \textup{(LO$^-$)} and $2$-connectivity of every member, and suppose
that neither a rainbow $C_{n-1}$ nor $C_n$ exists. Put
$h=\lfloor n/2\rfloor$. No rainbow Hamilton path containing a common
edge has both endpoint system degrees at least $h$.
\end{lemma}
\begin{proof}
For even $n$, Lemma~\ref{lem2.1} proves the assertion. For odd $n=2k+1$,
Lemma~\ref{lem:critical-count}\textup{(1)} forces both endpoints of
such a path $H=y_1\cdots y_n$ to have degree $k$ and forces equality
in that lemma. Write $r$ for the missing color and $w=y_{n-1}$.
If $\delta(w)\ge k$, Theorem~\ref{thm:prescribed-closure} closes the
prefix $H-y_n$, a contradiction. Thus $\delta(w)<k$.
An existing edge $y_1w$ would be common by \textup{(LO$^-$)} and would
close the same prefix in color $r$; hence it is absent in every layer.
In the equality partition in Lemma~\ref{lem:critical-count}\textup{(1)},
the position $n-2$ now forces $y_{n-2}y_n\in E(G_r)$. Delete $w$ and
join its two neighbors with color $r$. The resulting rainbow path on
$n-1$ vertices has endpoint degree sum $2k=n-1$, so
Theorem~\ref{thm:prescribed-closure} again gives a forbidden $C_{n-1}$.
\end{proof}

\begin{proof}[Proof of Theorem~\ref{thm:ore-almost-spanning}]
For $n=3$ every member is $K_3$. For $n=4$, the transversal Dirac theorem
applies. Suppose $n\ge5$ and neither long cycle exists, and put
$h=\lfloor n/2\rfloor$. By Theorem~\ref{thm2.3}, choose a rainbow
Hamilton path $P=v_1\cdots v_n$ maximizing
$\delta(v_1)+\delta(v_{n-1})$, with missing color $r$.
Theorem~\ref{thm:prescribed-closure} gives
\[
 \delta(v_1)+\delta(v_{n-1})\le n-2,
 \qquad \min\{\delta(v_1),\delta(v_{n-1})\}<h.
\]
If $\delta(v_1)<h$, apply Lemma~\ref{lem:trapped} to the prefix
$P-v_n$, using its path colors and $r$ as palette, with
$L=M=\{w:\delta(w)<h\}$. Condition~\textup{(1)} follows from the
absence of $C_{n-1},C_n$, and extremality gives condition~\textup{(2)}
after appending the unchanged final edge. Common edges between $L$ and
$M$ are supplied by $2(h-1)<n-1$ and \textup{(LO$^-$)}. This is
impossible. Hence $\delta(v_1)\ge h$ and $\delta(v_{n-1})<h$.

If $\delta(v_n)<h$, the last edge is common. Reverse $P$ and apply
the same lemma to its prefix, with the same $L=M$. In
condition~\textup{(2)}, append the fixed final edge to obtain a
Hamilton path containing the protected common edge and ending at $v_1$.
Lemma~\ref{lem:ore-high} forces its other endpoint below $h$, proving
that condition and giving another contradiction.

It remains that $\delta(v_1),\delta(v_n)\ge h$. As in the proof of
Theorem~\ref{thm2.4}, right rotations show that
\[
 I=\{i+2:v_iv_n\in E(G_r),\ 3\le i\le n-3\}\subseteq[5,n-1]
\]
consists of low-degree indices and has size at least $h-1$.
For $n=5$ the neighbor restriction is impossible. For $n\ge6$,
$h-1>\lceil(n-5)/2\rceil$, so $I$ has consecutive indices.
Their path edge is common by \textup{(LO$^-$)}, contradicting
Lemma~\ref{lem:ore-high} for the original $P$.
\end{proof}

\paragraph{Sharpness and scope.}
For odd $n=2k+1\ge5$, the diagonal system with every member equal to
$K_{k,k+1}$ satisfies \textup{(LO$^-$)}: the only nonedges lie within
a part, and their degree sums are $2k=n-1$ or $2k+2$.
It has no Hamilton cycle, so a Hamilton conclusion at threshold $n-1$
is false without an additional hypothesis.
For even $n=2k\ge6$, the diagonal system $K_{k-1,k+1}$ satisfies
the weaker implication with threshold $n-2$, but its longest cycle has
length $n-2$. Thus $n-1$ cannot be lowered uniformly in
Theorem~\ref{thm:ore-almost-spanning}.
It also prevents lowering both the local Ore and endpoint thresholds
to $n-2$ in a prescribed-set statement:
an alternating path using all $k-1$ vertices of the smaller part and
$k$ vertices of the larger part has $n-1$ vertices and endpoint degree
sum $n-2$, while its vertex set has odd size and spans no cycle.
This example fails \textup{(LO$^-$)} and therefore does not show that
the endpoint threshold alone is sharp when \textup{(LO$^-$)} is retained.
Neither endpoint sharpness under that retained hypothesis nor a separate
optimal local threshold for each parity is asserted here.

\section{The structural obstruction}\label{sec:structural-proof}

Throughout this section, every $G_i$ is $2$-connected, $\G$ satisfies
\textup{(RF)}, and no rainbow Hamilton cycle exists. For $n\in\{3,4\}$,
every $2$-connected $n$-vertex graph has minimum degree at least $n/2$, so the
transversal Dirac theorem of Joos and Kim~\cite{Joos2020} already yields a
rainbow Hamilton cycle. We may therefore assume $n\ge5$. By
Theorem~\ref{thm2.4}, $\G$ contains a rainbow cycle of length $n-1$.
Whenever
\[
C=v_1v_2\cdots v_{n-1}v_1
\]
is such a cycle, subscripts are interpreted modulo $n-1$. For
$a\in\mathbb Z$, let $\langle a\rangle$ be the unique element of $[n-1]$
congruent to $a$ modulo $n-1$. To avoid repeated brackets, an index arising
from the cyclic labeling of $C$ is written as $v_a$, $G_a$, or ``color $a$''
to mean $v_{\langle a\rangle}$, $G_{\langle a\rangle}$, or color
$\langle a\rangle$, respectively. The distinguished missing label $n$ is
never reduced modulo $n-1$.

The proof has three stages. Claims~1--4 establish separation properties for
vertices of system degree below $n/2$. Claims~5--6 then show that the vertex
omitted by a rainbow $(n-1)$-cycle has system degree at least $n/2$ and, more
importantly, that every vertex can occur as the omitted vertex of such a
cycle. Finally, we encode the available replacements along a fixed
$(n-1)$-cycle in an auxiliary digraph. Retaining the maximum-degree
differences and the color-degree surpluses strengthens its double count
to the degree-sequence bound.

\subsection{Low-degree separation}

\paragraph{Claim 1: consecutive low vertices force low endpoints.}
Let $P=x_1\cdots x_n$ be a rainbow Hamilton path containing consecutive
vertices $x_p,x_{p+1}$ of system degree below $n/2$. Both endpoints
then have system degree below $n/2$.
\begin{proof}
The edge $x_px_{p+1}$ is common by \textup{(RF)}. Suppose by symmetry
that $\delta(x_n)\ge n/2$. Lemma~\ref{lem2.1} gives
$\delta(x_1)<n/2$. Normalize the path colors to $1,\ldots,n-1$,
so color $n$ is missing. No closing edge lies in $G_n$.
Minimum degree two gives $x_1x_k\in E(G_n)$ for some $3\le k\le n-1$.
Rotate to $x_{k-1}\cdots x_1x_k\cdots x_n$.
If the original common edge is deleted, its exposed endpoint is one
of $x_p,x_{p+1}$ and is low. Otherwise the common edge is retained,
and Lemma~\ref{lem2.1} gives the same conclusion. Thus (RF), applied
to $x_1x_kx_{k-1}$, makes $x_1x_{k-1}$ common.
Descending rotations along these new common chords, each of which
retains a common edge, yield $x_1x_2\in E_{\cap}$.

Apply Lemma~\ref{lem:trapped} to the full path $P$, with palette $[n]$
and $L=M=\{w:\delta(w)<n/2\}$. Condition~\textup{(1)} is the absence
of a rainbow Hamilton cycle. Every path in condition~\textup{(2)}
retains the common edge $x_1x_2$ and ends at the high vertex $x_n$;
Lemma~\ref{lem2.1} therefore puts its first endpoint in $L$.
The fan lemma contradicts $2$-connectivity. The same reasoning after
reversal proves the assertion for the other endpoint.
\end{proof}

\paragraph{Claim 2: a high penultimate vertex forces an $(n-1)$-cycle.}
	Let $P=x_1x_2\cdots x_n$ be a rainbow Hamilton path such that
	\[
	\delta(x_{n-1})\ge\frac n2
	\qquad\text{and}\qquad
	d^c(x_1,x_n)\le2.
	\]
	Then $\G$ contains a rainbow cycle on
	$\{x_1,\ldots,x_{n-1}\}$.
	
	\begin{proof}
		Suppose otherwise. Globally permute the member labels so that
		$x_ix_{i+1}$ has color $i$. By Lemma~\ref{lem2.2},
		\[
		\delta(x_1)+\delta(x_{n-1})<n,
		\]
		so $\delta(x_1)<n/2$.
		
		Neither $x_1x_n$ nor $x_1x_{n-1}$ belongs to $G_n$, since these
		edges would give a rainbow cycle of length $n$ and $n-1$,
		respectively. Hence there exists $k\in[3,n-2]$ with
		$x_1x_k\in E(G_n)$. Applying Lemma~\ref{lem2.2} to the first
		$n-1$ vertices of
		\[
		x_{k-1}\cdots x_1x_k\cdots x_{n-1}x_n
		\]
		gives $\delta(x_{k-1})<n/2$, and therefore
		$x_1x_{k-1}\in E_{\cap}$ by (RF). Descending rotations yield
		\[
		\delta(x_i)<\frac n2,\qquad
		x_1x_i\in E_{\cap}
		\qquad(2\le i\le k-1).
		\]
		In particular, $\delta(x_2)<n/2$.
		
		Thus $P$ contains the consecutive low-degree vertices $x_1,x_2$.
		Claim~1 gives $\delta(x_n)<n/2$. Since
		$d^c(x_1,x_n)\le2$, (RF) yields $x_1x_n\in E_{\cap}$, which closes
		$P$ into a rainbow Hamilton cycle, a contradiction.
	\end{proof}

	\paragraph{Claim 3: low vertices are not consecutive on an $(n-1)$-cycle.}
	No rainbow $(n-1)$-cycle contains two consecutive vertices of system
	degree below $n/2$.
	
	\begin{proof}
		Let
		\[
		C=x_1x_2\cdots x_{n-1}x_1
		\]
		be a rainbow $(n-1)$-cycle, and suppose
		$\delta(x_i),\delta(x_{i+1})<n/2$. Let $y$ be the vertex outside
		$C$, and assume that $C$ uses colors $1,\ldots,n-1$.
		
		Since $G_n$ is $2$-connected, $d_{G_n}(y)\ge2$. Choose
		$x_j\in N_{G_n}(y)$ with $j\ne i$.
		Deleting $x_jx_{j+1}$ and adding $yx_j$ in color $n$ gives a
		rainbow Hamilton path
		\[
		yx_jx_{j-1}\cdots x_{j+1}
		\]
		still containing the low-degree edge $x_ix_{i+1}$. Claim~1 gives
		$\delta(y),\delta(x_{j+1})<n/2$. Since
		$yx_jx_{j+1}$ is a rainbow path of length two, (RF) gives
		$yx_{j+1}\in E_{\cap}$; assigning this edge color $j$ yields a
		rainbow Hamilton cycle, a contradiction.
	\end{proof}

	Choose a rainbow $(n-1)$-cycle
	\[
	C=v_1v_2\cdots v_{n-1}v_1
	\]
	so that the system degree of its omitted vertex $u$ is maximum.
	Globally permute the member labels so that $v_iv_{i+1}$ has color $i$;
	thus $C$ misses color $n$.
	
	\paragraph{Claim 4: no common edge joins two low vertices.}
	If $\delta(x),\delta(y)<n/2$, then $xy\notin E_{\cap}$.
	
	\begin{proof}
		Suppose $xy\in E_{\cap}$. If $x=u$ and $y=v_j$, delete the
		cycle edge $v_{j-1}v_j$ and assign its released color $j-1$ to the
		common edge $uv_j$. Then
		\[
		uv_jv_{j+1}\cdots v_{j-1}
		\]
		is a rainbow Hamilton path containing the consecutive low-degree
		vertices $u,v_j$. Claim~1 gives
		$\delta(v_{j-1})<n/2$, contradicting Claim~3.
		
		Thus let $x=v_i$ and $y=v_j$. By Claim~3, $v_i$ and $v_j$ are
		nonconsecutive on $C$. Applying Claim~3 to the two
		cycle edges immediately preceding them gives
		\[
		\delta(v_{i-1}),\delta(v_{j-1})\ge\frac n2.
		\]
		Assign color $n$ to $v_iv_j$ and delete
		$v_{i-1}v_i$ and $v_{j-1}v_j$. This gives a rainbow path
		\[
		P=v_{i-1}\cdots v_jv_i\cdots v_{j-1}
		=x_1x_2\cdots x_{n-1}
		\]
		missing colors $i-1$ and $j-1$, where
		\[
		x_1=v_{i-1},\qquad x_{n-1}=v_{j-1},
		\]
		and $x_s=v_j,x_{s+1}=v_i$ for some $s$. Hence
		\[
		\delta(x_1),\delta(x_{n-1})\ge\frac n2,
		\qquad
		\delta(x_s),\delta(x_{s+1})<\frac n2.
		\]
		
		We have
		\begin{equation}\label{eq:claim4-u}
			u\notin N_{G_{i-1}}(x_1)\cup N_{G_{j-1}}(x_{n-1}),
		\end{equation}
		for otherwise adjoining $u$ in the corresponding missing color gives a
		Hamilton path containing the low-degree edge $x_sx_{s+1}$; Claim~1 then
		contradicts the high degree of the other endpoint. Also
		\begin{equation}\label{eq:claim4-end}
			i-1,j-1\notin c(x_1x_{n-1}),
		\end{equation}
		for otherwise $P$ can be closed into a rainbow $(n-1)$-cycle still
		containing $x_sx_{s+1}$, contrary to Claim~3.
		
		Set
		\[
		\begin{gathered}I_1=\{t\in[n-3]:i-1\in c(x_1x_{t+1})\},\\
I_2=\{t\in[2,n-2]:j-1\in c(x_tx_{n-1})\}.
\end{gathered}
		\]
		By \eqref{eq:claim4-u} and \eqref{eq:claim4-end},
		\[
		|I_1|\ge\delta(x_1),\qquad
		|I_2|\ge\delta(x_{n-1}),
		\]
		so $|I_1|+|I_2|\ge n$. Since
		$I_1\cup I_2\subseteq[n-2]$, we have
		$|I_1\cap I_2|\ge2$. Choose
		$m\in I_1\cap I_2$ with $m\ne s$.
		
		Replacing $x_mx_{m+1}$ by
		$x_mx_{n-1}$ and $x_1x_{m+1}$ in colors $j-1$ and $i-1$ gives
		the rainbow $(n-1)$-cycle
		\[
		x_1\cdots x_mx_{n-1}x_{n-2}\cdots x_{m+1}x_1.
		\]
		Since $m\ne s$, it still contains the low-degree edge
		$x_sx_{s+1}$, contradicting Claim~3.
	\end{proof}
	
	Thus
	\begin{equation}\label{eq:low-distance}
		\delta(x),\delta(y)<\frac n2
		\quad\Longrightarrow\quad
		d^c(x,y)>2.
	\end{equation}
	
	\subsection{The omitted vertex}

	From this point on, subscripts on the fixed cycle
	$C=v_1\cdots v_{n-1}v_1$ are read modulo $n-1$. When a displayed
	vertex sequence follows two oriented arcs of $C$, these are the indicated
	cyclic arcs. The index restrictions in each argument ensure that the two
	arcs have disjoint interiors and, together with $u$, cover every vertex
	exactly once.

\paragraph{Claim 5: the omitted vertex has high system degree.}
	The omitted vertex satisfies $\delta(u)\ge n/2$.
	
	\begin{proof}
		Suppose $\delta(u)<n/2$. By the choice of $u$,
		\begin{equation}\label{eq:max-u}
			\delta(v)\ge\frac n2
			\quad\Longrightarrow\quad
			\text{there is no rainbow $(n-1)$-cycle on }V\setminus\{v\}.
		\end{equation}
		
		Choose $v_i\in N_{G_n}(u)$. Since $uv_iv_{i-1}$ is rainbow,
		\eqref{eq:low-distance} gives
		$\delta(v_{i-1})\ge n/2$.
		
		We first show that $\delta(v_{i-2})<n/2$. Otherwise
		\[
		uv_iv_{i+1}\cdots v_{i-2}v_{i-1}
		\]
		is a rainbow Hamilton path whose first vertex is $u$, whose last
		vertex is $v_{i-1}$, and whose penultimate vertex $v_{i-2}$ has
		system degree at least $n/2$. Since $d^c(u,v_{i-1})\le2$,
		Claim~2 applies and gives a rainbow $(n-1)$-cycle on
		$V\setminus\{v_{i-1}\}$, contrary to \eqref{eq:max-u}.
		
		Hence $\delta(v_{i-2})<n/2$, and \eqref{eq:low-distance} gives
		$d^c(u,v_{i-2})>2$. Claim~3 then yields
		$\delta(v_{i-3})\ge n/2$. Also,
		\[
		\delta(v_{i-4})\ge\frac n2,
		\]
		for otherwise $v_{i-4}v_{i-3}v_{i-2}$ contradicts
		\eqref{eq:low-distance}.
		
		Since $d_{G_{i-1}}(u)\ge2$, choose $j\ne i$ with
		$uv_j\in E(G_{i-1})$. Then
		$uv_jv_{j-1}$ and \eqref{eq:low-distance} give
		$\delta(v_{j-1})\ge n/2$. Also
		\[
		j\notin\{i-1,i-2,i-3\},
		\]
		since otherwise $d^c(u,v_{i-2})\le2$. If $n=5$, the three indices
		$i-1,i-2,i-3$ are precisely all indices of the cycle distinct from
		$i$, contradicting the choice $j\ne i$. Hence, for the remainder of
		this claim, we may assume $n\ge6$.
		
		Consider
		\[
		P=v_{j-1}\cdots v_i\,u\,v_j\cdots
		v_{i-4}v_{i-3}v_{i-2}v_{i-1},
		\]
		where $v_iu$ has color $n$ and $uv_j$ has color $i-1$.
		Thus $P$ misses color $j-1$. Write
		\[
		P=x_1x_2\cdots x_n,
		\]
		where
		\[
		x_1=v_{j-1},\quad x_n=v_{i-1},\quad
		x_{n-3}=v_{i-4},\quad
		x_{n-2}=v_{i-3},\quad
		x_{n-1}=v_{i-2}.
		\]
		
		We have
		\[
		j-1\notin c(x_1x_{n-1})\cup c(x_1x_n),
		\]
		since the first possibility gives a rainbow $(n-1)$-cycle omitting $x_n$,
		where $\delta(x_n)\ge n/2$,  while the second gives a rainbow
		Hamilton cycle.
		
		Define
		\[
		I_1=\{t\in[n-3]:j-1\in c(x_1x_{t+1})\},\qquad
		I_2=\{t\in[n-1]:i-3\in c(x_tx_n)\}.
		\]
		The two displayed nonadjacencies show that every $G_{j-1}$-neighbor of
		$x_1$ lies among $x_2,\ldots,x_{n-2}$, and therefore
		$|I_1|=d_{G_{j-1}}(x_1)$. Also
		$|I_2|=d_{G_{i-3}}(x_n)$. Hence
		\[
		|I_1|\ge\delta(x_1)\ge\frac n2,\qquad
		|I_2|\ge\delta(x_n)\ge\frac n2.
		\]
		Viewing both sets as subsets of $[n-1]$, inclusion--exclusion gives
		$I_1\cap I_2\ne\varnothing$; choose
		$m\in I_1\cap I_2$. Since $m\le n-3$, there are two
		possibilities.
		
		If $m=n-3$, assign color $i-3$ to $x_{n-3}x_n$ and the missing
		color $j-1$ to $x_1x_{n-2}$. Relative to $P$, the path below omits
		the two edges $x_{n-3}x_{n-2}$ and $x_{n-2}x_{n-1}$. The latter is
		the unique color-$(i-3)$ edge of $P$, while color $j-1$ was missing
		from $P$. Thus the two new edges introduce no repeated color, and
		\[
		Q=x_{n-1}x_nx_{n-3}x_{n-4}\cdots x_1x_{n-2}
		\]
		is a rainbow Hamilton path with first vertex $x_{n-1}$, last vertex
		$x_{n-2}$, and penultimate vertex $x_1$. We have
		$\delta(x_1)\ge n/2$ and $d^c(x_{n-1},x_{n-2})=1$, so Claim~2
		gives a rainbow $(n-1)$-cycle on
		$V\setminus\{x_{n-2}\}=V\setminus\{v_{i-3}\}$, contradicting
		\eqref{eq:max-u}.
		
		If $m<n-3$, assign color $i-3$ to $x_mx_n$ and the missing color
		$j-1$ to $x_1x_{m+1}$. This time the omitted edges of $P$ are
		$x_mx_{m+1}$ and $x_{n-2}x_{n-1}$. Again, the latter is the unique color-$(i-3)$ edge of $P$, while color
		$j-1$ was absent from $P$. Hence the two inserted edges have colors
		distinct from all retained edges, and
		\[
		Q=x_{n-1}x_nx_mx_{m-1}\cdots x_1
		x_{m+1}\cdots x_{n-3}x_{n-2}
		\]
		is a rainbow Hamilton path with first vertex $x_{n-1}$, last vertex
		$x_{n-2}$, and penultimate vertex $x_{n-3}=v_{i-4}$. Since
		$\delta(x_{n-3})\ge n/2$ and $d^c(x_{n-1},x_{n-2})=1$, Claim~2
		again gives a rainbow $(n-1)$-cycle on
		$V\setminus\{x_{n-2}\}=V\setminus\{v_{i-3}\}$, contradicting
		\eqref{eq:max-u}. Therefore $\delta(u)\ge n/2$.
	\end{proof}

	\paragraph{Claim 6: every vertex can be omitted.}
	For every $w\in V$, $\G$ contains a rainbow cycle on
	$V\setminus\{w\}$.

\begin{proof}
	The fixed cycle $C$ proves the assertion for $w=u$. For a cycle vertex
	$w=v_r$, cyclically relabel the positions of $C$ so that $w$ becomes
	$v_2$, and apply the same cyclic permutation to the member labels used on
	$C$, fixing the missing label $n$. This global relabeling preserves system
	degrees, \textup{(RF)}, and the existence of rainbow cycles. It therefore
	suffices to consider $w=v_2$.

	Suppose otherwise. By Claim~5, $\delta(u)\ge n/2$. Set
\[
A_2=\{j:uv_j\in E(G_2)\},\qquad
B_n=\{j:uv_{j+1}\in E(G_n)\}.
\]
Since $|A_2|,|B_n|\ge n/2$ and both are subsets of the
$(n-1)$-element set $[n-1]$, their intersection is nonempty. Choose
$j\in A_2\cap B_n$. Thus
\[
uv_j\in E(G_2),\qquad uv_{j+1}\in E(G_n).
\]
If $j=2$, replace the color-$2$ edge $v_2v_3$ of $C$ by the
two-edge path $v_2uv_3$, using colors $2$ and $n$, respectively.
This gives a rainbow Hamilton cycle, so $j\ne2$.

	Applying the same pigeonhole argument to colors $1$ and $n$, choose
$k$ such that
\[
uv_k\in E(G_1),\qquad uv_{k+1}\in E(G_n).
\]
If $k=1$, replace the color-$1$ edge $v_1v_2$ by
$v_1uv_2$, using colors $1$ and $n$. Again this gives a rainbow Hamilton cycle, so $k\ne1$.

If $j=1$, then
\[
v_3v_4\cdots v_{n-1}v_1uv_2
\]
is a rainbow Hamilton path: use color $2$ on $v_1u$ and color $n$
on $uv_2$, while the retained cycle edges keep their colors. Its
penultimate vertex is $u$, which has system degree at least $n/2$,
and $d^c(v_3,v_2)=1$. Claim~2 therefore gives a rainbow cycle on
$V\setminus\{v_2\}$, a contradiction. Thus $j\ne1$.

If $k=2$, the Hamilton path
\[
v_1v_{n-1}v_{n-2}\cdots v_3uv_2
\]
uses color $n$ on $v_3u$ and color $1$ on $uv_2$, while the
retained cycle edges keep their colors. Its penultimate vertex is
again $u$, and $d^c(v_1,v_2)=1$, so Claim~2 gives a rainbow cycle on
$V\setminus\{v_2\}$, a contradiction. Thus $k\ne2$.

By \eqref{eq:low-distance},
\[
\max\{\delta(v_1),\delta(v_3)\}\ge\frac n2.
\]
Suppose first that $\delta(v_1)\ge n/2$. If $3\le j\le n-2$, then
\[
v_3v_4\cdots v_j\,u\,v_{j+1}\cdots v_{n-1}v_1v_2
\]
is a rainbow Hamilton path: the cycle edges of colors $2$ and $j$
are omitted, color $2$ is used on $v_ju$, and color $n$ is used on
$uv_{j+1}$. Thus its penultimate vertex is $v_1$. If $j=n-1$, use
instead
\[
v_3v_4\cdots v_{n-1}\,u\,v_1v_2.
\]
Here the omitted cycle-edge colors are $2$ and $n-1$,
$v_{n-1}u$ uses color $2$, and $uv_1$ uses color $n$; again the
penultimate vertex is $v_1$. In both cases $d^c(v_3,v_2)=1$, so
Claim~2 gives a rainbow cycle on $V\setminus\{v_2\}$, a
contradiction.

Therefore $\delta(v_3)\ge n/2$. We already know $k\notin\{1,2\}$.
If $3\le k\le n-2$, then
\[
v_1v_{n-1}\cdots v_{k+1}\,u\,v_kv_{k-1}\cdots v_3v_2
\]
is a rainbow Hamilton path: the cycle edges of colors $1$ and $k$
are omitted, color $n$ is used on $v_{k+1}u$, and color $1$ is used
on $uv_k$. Thus its penultimate vertex is $v_3$. If $k=n-1$, use
instead
\[
v_1\,u\,v_{n-1}v_{n-2}\cdots v_3v_2,
\]
where $v_1u$ uses color $n$ and $uv_{n-1}$ uses color $1$; again
the penultimate vertex is $v_3$. In both cases
$d^c(v_1,v_2)=1$, so Claim~2 gives a rainbow cycle on
$V\setminus\{v_2\}$, a contradiction. The claim follows.
\end{proof}
	
\subsection{The auxiliary digraph and the degree-sum bound}

\begin{proof}[Proof of Theorem~\ref{thm:structural}]
For the fixed rainbow $(n-1)$-cycle $C$ and its omitted vertex $u$,
\eqref{eq:low-distance} holds, $\delta(u)\ge n/2$, and every vertex can be
omitted by a rainbow $(n-1)$-cycle.

	\medskip
	\noindent\emph{The common edge set is empty.}
	Suppose $xy\in E_{\cap}$. By Claim~4, we may assume
	$\delta(x)\ge n/2$. By Claim~6, there is a rainbow cycle
	\[
	C'=y_1y_2\cdots y_{n-1}y_1
	\]
	on $V\setminus\{x\}$ with $y_1=y$.
	
	Delete $y_{n-1}y_1$ and transfer its color to the common edge $xy_1$.
	The resulting Hamilton path
	\[
	xy_1y_2\cdots y_{n-1}
	\]
	contains a common edge, so Lemma~\ref{lem2.1} gives
	$\delta(y_{n-1})<n/2$. Deleting $y_1y_2$ instead and assigning
	its color to the common edge $xy_1$ gives the Hamilton path
	$xy_1y_{n-1}y_{n-2}\cdots y_2$; another application of
	Lemma~\ref{lem2.1} gives $\delta(y_2)<n/2$. But
	$y_2y_1y_{n-1}$ is a rainbow path of length two, contradicting
	\eqref{eq:low-distance}. Therefore
	\begin{equation}\label{eq:Ecap-empty}
		E_{\cap}=\varnothing.
	\end{equation}
	\medskip
	\noindent\emph{The degree-sum bound.}
	Retain the chosen rainbow $(n-1)$-cycle
	\[
	C=v_1v_2\cdots v_{n-1}v_1
	\]
	with omitted vertex $u$, and assume that $v_iv_{i+1}$ has color $i$
	for every $i\in[n-1]$. By the choice of $u$ and Claim~6,
\[
 \Delta=\delta(u)=\max_{v\in V}\delta(v).
\]
Put $d=d_{G_n}(u)\ge\Delta$ and $N=N_{G_n}(u)$.

For completeness, $\Delta\le n-3$. Indeed, $d\ge n/2$ ensures two
consecutive vertices $v_i,v_{i+1}\in N$. Neither $uv_i$ nor
$uv_{i+1}$ can lie in $G_i$, since either would insert $u$ into $C$
using colors $i,n$. Thus $d_{G_i}(u)\le n-3$.

	Define an auxiliary digraph $D$ on $V$ by putting an arc
	$v_i\to z$ whenever
	\[
	z\ne v_{i+1}
	\qquad\text{and}\qquad
	v_iz\in E(G_i),
	\]
	for $i\in[n-1]$. Thus an arc from $v_i$ records an alternative
	color-$i$ edge that can replace the cycle edge $v_iv_{i+1}$.
	
	By definition,
	\begin{equation}\label{eq:D1}
		d_D^+(u)=0.
	\end{equation}
	For every $i\in[n-1]$,
	\begin{equation}\label{eq:D2}
		d_D^+(v_i)
		=d_{G_i}(v_i)-1
		\ge\delta(v_i)-1.
	\end{equation}
	
	For the omitted vertex $u$, set
	\[
	I_1=\{i\in[n-1]:v_i\to u\in A(D)\},\qquad
	I_2=\{i\in[n-1]:uv_{i+1}\in E(G_n)\}.
	\]
	The sets $I_1$ and $I_2$ are disjoint: for
	$i\in I_1\cap I_2$, replacing $v_iv_{i+1}$ by
	$v_iu$ and $uv_{i+1}$, in colors $i$ and $n$, respectively,
	gives a rainbow Hamilton cycle. Hence
	\[
	d_D^-(u)+d_{G_n}(u)=|I_1|+|I_2|\le n-1,
	\]
	and therefore
	\begin{equation}\label{eq:D3}
		d_D^-(u)\le n-1-d.
	\end{equation}
	
	For $v_i\in V(C)$, first suppose that $v_i\in N_{G_n}(u)$. The path
	\[
	P=uv_iv_{i+1}\cdots v_{i-1},
	\]
	where $uv_i$ receives color $n$, is a rainbow Hamilton path missing
	color $i-1$. Define
	\[
	J_1=\{j\in[n-1]:uv_{j+1}\in E(G_{i-1})\},\qquad
	J_2=\{j\in[n-1]:v_j\in N_D^-(v_{i-1})\}.
	\]
	The sets $J_1$ and $J_2$ are disjoint. Otherwise, choose
	$j\in J_1\cap J_2$. Then
	\[
	uv_{j+1}\in E(G_{i-1}),
	\qquad
	v_jv_{i-1}\in E(G_j),
	\]
	where the second relation follows from the definition of $D$.
	Deleting the color-$j$ edge $v_jv_{j+1}$ and adding
	$v_jv_{i-1}$ in color $j$ and $uv_{j+1}$ in the missing color
	$i-1$ yields a rainbow Hamilton cycle, a contradiction.
	
	Also, $i-2\notin J_1\cup J_2$. Membership in $J_1$ would give
	$uv_{i-1}\in E(G_{i-1})$, which together with
	$uv_i\in E(G_n)$ gives a rainbow Hamilton cycle, while
	$i-2\notin J_2$ because $v_{i-2}v_{i-1}$ is precisely the cycle
	edge excluded from the definition of $D$. Thus
	\[
	|J_1|+|J_2|\le n-2.
	\]
	The cyclic map $j\mapsto v_{j+1}$ is a bijection from $[n-1]$
	onto $V(C)$, so $|J_1|=d_{G_{i-1}}(u)$. Likewise every arc entering
	$v_{i-1}$ has its tail on $C$ because $u$ has outdegree zero, and the
	map $j\mapsto v_j$ gives $|J_2|=d_D^-(v_{i-1})$. Hence
	\[
	|J_1|=d_{G_{i-1}}(u)\ge\delta(u),
	\qquad
	|J_2|=d_D^-(v_{i-1}),
	\]
	and therefore
	\[
	\delta(u)+d_D^-(v_{i-1})\le n-2.
	\]
	Retaining the actual color degree of $u$ gives
\begin{equation}\label{eq:D4}
 d_D^-(v_{i-1})\le n-2-d_{G_{i-1}}(u)
 \qquad\text{whenever }v_i\in N.
\end{equation}

	Now suppose that $v_i\notin N_{G_n}(u)$. By Claim~5,
	$\delta(u)\ge n/2$, and hence
	\[
	d_{G_{i-1}}(u)+d_{G_n}(u)\ge2\delta(u)\ge n.
	\]
	Set
	\[
	A=\{j\in[n-1]:v_{j-1}\in N_{G_{i-1}}(u)\},\qquad
	B=\{j\in[n-1]:v_j\in N_{G_n}(u)\}.
	\]
	Since $|A|+|B|\ge n$ and $A,B\subseteq[n-1]$, the intersection
	$A\cap B$ is nonempty; choose $j\in A\cap B$. Then
	\[
	uv_{j-1}\in E(G_{i-1}),
	\qquad
	uv_j\in E(G_n).
	\]
	Moreover, $j\ne i$ because $v_i\notin N_{G_n}(u)$.
	
	Thus
	\[
	P=
	v_iv_{i+1}\cdots v_{j-1}
	u
	v_jv_{j+1}\cdots v_{i-1}
	\]
	is a rainbow Hamilton path, where $uv_{j-1}$ has color $i-1$ and
	$uv_j$ has color $n$. Hence $P$ misses color $j-1$.
	
	Write
	\[
	P=x_1x_2\cdots x_n,
	\qquad
	x_1=v_i,\quad x_n=v_{i-1},
	\]
	and define
	\[
	\begin{gathered}
 K_1=\{k\in[n-2]:j-1\in c(x_1x_{k+1})\},\\
 K_2=\{k\in[n-2]:x_k\in N_D^-(v_{i-1})\}.
\end{gathered}
	\]
	We have
	\[
	|K_1|+|K_2|\le n-1.
	\]
	Indeed, if the inequality failed, then $K_1,K_2\subseteq[n-2]$ would give
	\[
	|K_1\cap K_2|\ge2.
	\]
	Now $x_k=u$ is impossible for $k\in K_2$ by
	\eqref{eq:D1}, and at most one index corresponds to
	$x_k=v_{j-1}$. Hence we may choose
	$k\in K_1\cap K_2$ such that
	\[
	x_k\ne u,v_{j-1}.
	\]
	
	Write $x_k=v_r$. Since $x_k$ is neither $u$ nor $v_{j-1}$, the
	edge following $x_k$ on $P$ is the original cycle edge
	\[
	x_kx_{k+1}=v_rv_{r+1},
	\]
	and hence has color $r$. On the other hand, $k\in K_2$ implies that
	$v_r\to v_{i-1}$ is an arc of $D$, so
	\[
	x_kx_n=v_rv_{i-1}\in E(G_r).
	\]
	Also, $k\in K_1$ gives
	$x_1x_{k+1}\in E(G_{j-1})$. Therefore, deleting
	$x_kx_{k+1}$ and adding $x_kx_n$ in color $r$ and
	$x_1x_{k+1}$ in the missing color $j-1$ gives the rainbow Hamilton
	cycle
	\[
	x_1x_2\cdots x_kx_nx_{n-1}\cdots x_{k+1}x_1,
	\]
	a contradiction. Hence $|K_1|+|K_2|\le n-1$.
	
	Furthermore,
	\[
	j-1\notin c(x_1x_n),
	\]
	since otherwise $P$ could be closed into a rainbow Hamilton cycle.
	Moreover, $x_{n-1}\notin N_D^-(v_{i-1})$. If $j\ne i-1$, then
	$x_{n-1}=v_{i-2}$ and $v_{i-2}v_{i-1}$ is the cycle edge excluded
	from the definition of $D$; if $j=i-1$, then $x_{n-1}=u$, which has
	outdegree zero in $D$ by \eqref{eq:D1}. Consequently, every $G_{j-1}$-neighbor of $x_1=v_i$ occurs among
	$x_2,\ldots,x_{n-1}$, which are exactly the positions indexed by
	$K_1$; hence $|K_1|=d_{G_{j-1}}(v_i)$. Similarly, all tails of arcs
	entering $v_{i-1}$ occur among $x_1,\ldots,x_{n-1}$, and the only
	position omitted from $K_2$, namely $x_{n-1}$, is not such a tail.
	Thus
	\[
	|K_1|=d_{G_{j-1}}(v_i)\ge\delta(v_i),
	\qquad
	|K_2|=d_D^-(v_{i-1}),
	\]
	and hence
	\begin{equation}\label{eq:D5}
		d_D^-(v_{i-1})
		\le n-1-\delta(v_i)
		\qquad
		\text{whenever }v_i\notin N_{G_n}(u).
	\end{equation}

We now count the arcs of $D-u$ in two ways, keeping the discarded
surpluses. Define
\[
\begin{gathered}
 R=\sum_{i=1}^{n-1}\bigl(d_{G_i}(v_i)-\delta(v_i)\bigr),\qquad
 E=\sum_{v_i\in N}\bigl(d_{G_{i-1}}(u)-\Delta\bigr),\\
 T=\sum_{v_i\in N}\bigl(\Delta-\delta(v_i)\bigr).
\end{gathered}
\]
All three quantities are nonnegative. Since no arc leaves $u$, deleting
$u$ removes exactly the arcs entering it. By \eqref{eq:D2} and
\eqref{eq:D3},
\begin{align}
 |A(D-u)|
 &=\sum_{i=1}^{n-1}\bigl(d_{G_i}(v_i)-1\bigr)-d_D^-(u)\notag\\
 &\ge S-\Delta-(n-1)+R-(n-1-d)\notag\\
 &=S-2(n-1)+(d-\Delta)+R.\label{eq:D6}
\end{align}
On the other hand, summing \eqref{eq:D4} over $N$ and
\eqref{eq:D5} over its complement on the cycle gives
\begin{align}
 |A(D-u)|
 &\le\sum_{v_i\in N}\bigl(n-2-d_{G_{i-1}}(u)\bigr)
    +\sum_{v_i\notin N}\bigl(n-1-\delta(v_i)\bigr)\notag\\
 &=(n-1)^2-S-(d-\Delta)-T-E.\label{eq:D7}
\end{align}
Consequently the fixed cycle provides the more precise certificate
\begin{equation}\label{eq:full-slack-bound}
 2S+2(d-\Delta)+T+R+E\le n^2-1.
\end{equation}
This statement retains both color-degree information and the degree
gaps between $u$ and its missing-layer neighbors.

To obtain a bound depending only on the system-degree sequence, recall
that $|N|=d\ge\Delta$ and $u\notin N$. Among all sets of $d$ other
vertices, the sum of the deficits $\Delta-\delta(v)$ is minimized by
the $d$ largest degrees excluding one occurrence of $\Delta$.
Hence
\[
 T\ge\sum_{j=2}^{d+1}(\Delta-a_j)
   \ge\sum_{j=2}^{\Delta+1}(\Delta-a_j)=\Gamma.
\]
Dropping only the remaining nonnegative terms in
\eqref{eq:full-slack-bound} yields
\[
 2S+\Gamma\le n^2-1.
\]
Together with \eqref{eq:Ecap-empty}, this proves
Theorem~\ref{thm:structural}.
\end{proof}

\subsection{Strict improvement over the aggregate criterion}

The correction is realized by graph systems that satisfy neither of the
previous aggregate-degree and common-edge sufficient conditions.

\begin{proposition}\label{prop:strict-family}
For every $n=2t\ge10$, there is a system of $n$ graphs, each
$2$-connected, satisfying \textup{(RF)}, such that
\[
 E_{\cap}=\varnothing,\qquad
 S=2t^2-\lfloor t/2\rfloor,\qquad \Gamma=t.
\]
In particular, Corollary~\ref{cor:degree-sequence} applies, whereas
Corollaries~\ref{cor:total-degree-ham} and \ref{cor:common-edge-ham}
do not. The required aggregate degree is lower by
$\lfloor n/4\rfloor$ than the earlier scalar threshold on this family.
\end{proposition}
\begin{proof}
Put $g=\lfloor t/2\rfloor$, $b=2t-2-g$, and
\[
 p=\lfloor b/2\rfloor,\qquad q=\lceil b/2\rceil.
\]
Then $2\le p\le q\le t-2$ and $p+q+2=2t-g\le2t-2$.
Let $U$ be a set of $2t-2$ vertices, with distinguished vertices
$h_1,h_2$, and let $M_0,\ldots,M_{2t-4}$ be a $1$-factorization of
$K_U$. Explicitly, take $U=\{\infty\}\cup\mathbb Z_{2t-3}$ and
\[
 M_r=\{\{\infty,r\}\}\cup
     \{\{r+j,r-j\}:1\le j\le t-2\},
\]
with arithmetic modulo $2t-3$; take $h_1=\infty$ and $h_2=0$.
For each color $i\in[2t]$, put $r_i=i-1$ modulo $2t-3$ and let
\[
 F_i^0=M_{r_i}\cup M_{r_i+1}\cup\cdots\cup M_{r_i+t-5}.
\]
This is $(t-4)$-regular, and $\bigcup_iF_i^0=E(K_U)$, since every
factor occurs as the first factor of some $F_i^0$.

Write $W=U\setminus\{h_1,h_2\}$. The graph $K_W-F_i^0[W]$ has
order $2t-4$ and minimum degree at least $t-1$, so it has a perfect
matching $Q_i$, for example by the classical Dirac theorem.
Set $F_i=F_i^0\cup Q_i$ and $H_i=K_U-F_i$. Then
\[
 d_{H_i}(h_1)=d_{H_i}(h_2)=t+1,\qquad
 d_{H_i}(w)=t\quad(w\in W).
\]
Moreover, $\bigcap_iE(H_i)=\varnothing$.

Choose disjoint sets $A,B\subseteq U$ of sizes $p+1,q+1$, and list
them as $A=\{\alpha_0,\ldots,\alpha_p\}$ and
$B=\{\beta_0,\ldots,\beta_q\}$. Add vertices $x,y$.
The member $G_i$ consists of $H_i$ together with the edges
\[
 x\alpha_j\quad(j\ne i-1\pmod{p+1}),\qquad
 y\beta_j\quad(j\ne i-1\pmod{q+1});
\]
no other edges are added. Every possible attachment edge is absent
in at least one member, because $2t\ge p+1,q+1$.
Consequently $E_{\cap}=\varnothing$ and the system degrees are
\[
 t+1,\ t+1,\ \underbrace{t,\ldots,t}_{2t-4\text{ entries}},\ p,\ q.
\]
Each $H_i$ is $2$-connected: after one vertex is deleted, its minimum
degree is at least $t-1$ on $2t-3$ vertices, which forces connectivity.
Adding $x,y$ with at least two neighbors each preserves
$2$-connectivity. The only vertices below $n/2=t$ are $x,y$.
Their union neighborhoods lie in the disjoint sets $A,B$, so they
have rainbow distance greater than two. Thus \textup{(RF)} holds.

Finally,
\[
 S=2(t+1)+(2t-4)t+p+q=2t^2-g,
 \qquad \Delta=t+1,
\]
and the top $\Delta+1$ degrees consist of the two copies of $t+1$
and $t$ copies of $t$. Hence $\Gamma=t$, and
\[
 2S+\Gamma=4t^2+t-2\lfloor t/2\rfloor\ge n^2.
\]
Indeed $S=\lceil(n^2-\Gamma)/2\rceil$, so the new sufficient
inequality is met at its integer threshold. The original scalar
criterion requires $S\ge2t^2$, which fails.
\end{proof}

\section{The local Ore theorem}\label{sec:local-ore-proof}

\begin{proof}[Proof of Theorem~\ref{thm:local-ore}]
	Suppose for a contradiction that $\G$ has no rainbow Hamilton cycle. Since
	\[
	\delta(u),\delta(v)<\frac n2
	\quad\Longrightarrow\quad
	\delta(u)+\delta(v)<n,
	\]
	the hypothesis of Theorem~\ref{thm:local-ore} implies (RF). Hence
	Theorem~\ref{thm:structural} gives
	\begin{equation}\label{eq:structural-bound}
		E_{\cap}=\varnothing
		\qquad\text{and}\qquad
		\sum_{v\in V}\delta(v)\le\frac{n^2-1}{2}.
	\end{equation}
	
	By Theorem~\ref{thm2.4}, $\G$ contains a rainbow $(n-1)$-cycle
	\[
	C=v_1v_2\cdots v_{n-1}v_1.
	\]
	Let $v_n$ be the remaining vertex. Globally permute the member labels so
	that $C$ uses colors $1,\ldots,n-1$. Since $G_n$
	is connected,
	$v_n$ has a $G_n$-neighbor on $C$; by cyclic reindexing, assume
	\[
	v_1v_n\in E(G_n).
	\]
	Set $v_{n+1}=v_1$. Then
	\begin{equation}\label{eq:local-dist-cycle}
		d^c(v_i,v_{i+1})\le2
		\qquad(i\in[n]).
	\end{equation}
	For $1\le i\le n-2$, this follows directly from the cycle edges. For
	$i=n-1$, the path
	\[
	v_{n-1}v_1v_n
	\]
	is rainbow; for $i=n$, the vertices $v_n$ and $v_1$ are adjacent in
	$G_n$.
	
	Since $E_{\cap}=\varnothing$, the hypothesis of
	Theorem~\ref{thm:local-ore} and \eqref{eq:local-dist-cycle} imply
	\[
	\delta(v_i)+\delta(v_{i+1})\ge n
	\qquad(i\in[n]).
	\]
	Summing these $n$ inequalities gives
	\[
	2\sum_{v\in V}\delta(v)
	=
	\sum_{i=1}^n
	\bigl(\delta(v_i)+\delta(v_{i+1})\bigr)
	\ge n^2,
	\]
	contradicting \eqref{eq:structural-bound}. The proof is complete.
\end{proof}

\section{Concluding remarks and open problems}\label{sec:concluding}

Theorem~\ref{thm:structural} isolates the obstruction left by the local
Fan-type condition. Under \textup{(RF)}, a non-Hamiltonian system must have no
common edge and must lie below the aggregate-degree threshold in
Corollary~\ref{cor:total-degree-ham}. It remains open whether both restrictions
can occur simultaneously when every member is $2$-connected; equivalently,
whether \textup{(RF)} alone already forces a rainbow Hamilton cycle.

The degree-sequence correction in \eqref{eq:gamma-definition} gives a
strictly stronger sufficient inequality whenever the top $\Delta+1$
system degrees are not equal. The finer cycle certificate
\eqref{eq:full-slack-bound} also records color-degree surpluses.
Neither inequality is asserted to be optimal. In particular, no
non-Hamiltonian example satisfying \textup{(RF)} is supplied, so the
remaining obstruction should not be treated as a realized extremal family.

The critical Ore results have a different status. At threshold $n-1$,
both the general long-cycle conclusion and prescribed-set closure are
proved. Even-order complete bipartite examples prevent lowering the
local threshold for the long-cycle conclusion uniformly over all orders.
They do not establish endpoint sharpness while retaining
\textup{(LO$^-$)}. Odd-order complete bipartite
examples prevent upgrading the long-cycle conclusion to Hamiltonicity.
The optimal parity-specific refinements and a structural classification
of non-Hamiltonian systems under \textup{(LO$^-$)} remain open here.

The common-edge requirement also deserves attention. It is not merely a
technical device: a common edge can absorb any color released by a
rotation, which is the key source of recoloring flexibility in our proofs. It
would be interesting to replace this requirement by a weaker local condition
that retains enough flexibility for rotation--extension while still reducing
to the classical Fan or Ore condition in diagonal systems.

\section*{Declarations}
\paragraph{Funding.}
This work was supported by the National Natural Science Foundation of China
(Grant No.~12171088).
\paragraph{AI-assisted tools.}
ChatGPT (OpenAI) was used during manuscript revision to assist with
manuscript organization, language and LaTeX editing, suggestions for
mathematical arguments and proof checking, and drafting auxiliary code for
finite checks of selected examples. AI assistance does not constitute
independent verification of the mathematical results. The authors assume
responsibility for all content.
\paragraph{Data availability.}
Data sharing is not applicable to this article because no datasets were
generated or analyzed in the course of this study.

\end{document}